\documentclass[12pt]{amsart}
\usepackage{a4wide,fullpage}  
\usepackage{amssymb}
\usepackage{amsthm}
\usepackage{amsmath,cancel}
\usepackage{amscd,enumitem}
\usepackage{verbatim}
\usepackage{float}
\usepackage{color,mdframed}
\usepackage{bbm}
\usepackage[mathscr]{eucal}
\usepackage[all]{xy}
\usepackage{bm}
\usepackage{graphicx}
\usepackage{hyperref}
\usepackage{comment}
\usepackage{calligra}
\usepackage{stmaryrd}
\DeclareMathAlphabet{\mathcalligra}{T1}{calligra}{m}{n}

\theoremstyle{plain}
\newtheorem{theorem}{Theorem}
\newtheorem{corollary}[theorem]{Corollary}
\newtheorem{lemma}[theorem]{Lemma}
\newtheorem{proposition}[theorem]{Proposition}
\newtheorem*{theorem*}{Theorem}
\theoremstyle{definition}

\theoremstyle{remark}
\newtheorem*{remark}{Remark}
\newtheorem*{remarks}{Remarks}

\newcommand{\Z}{\mathbb{Z}}
\newcommand{\N}{\mathbb{N}}

\newcommand{\GL}{\operatorname{GL}}

\renewcommand{\pmod}[1]{\  \,  \left(  \operatorname{mod} \,  #1 \right)}

\DeclareMathAlphabet{\mathpzc}{OT1}{pzc}{m}{it}

\numberwithin{equation}{section}
\numberwithin{theorem}{section}
\numberwithin{table}{section}
\allowdisplaybreaks

\newcommand{\floor}[1]{\left\lfloor #1 \right\rfloor}
\newcommand{\ceil}[1]{\left\lceil #1 \right\rceil}

\begin{document}
\author{Soumyarup Banerjee}
\address{Department of Mathematics, University of Hong Kong, Pokfulam, Hong Kong}
\email{soumya.tatan@gmail.com}
\author{Manav Batavia}
\address{Department of Mathematics, Indian Institute of Technology Bombay, Powai, Mumbai, Maharashtra 400076, India}
\email{manavbatavia@gmail.com}
\author{Ben Kane}
\address{Department of Mathematics, University of Hong Kong, Pokfulam, Hong Kong}
\email{bkane@hku.hk}
\author{Muratzhan Kyranbay}
\address{Department of Mathematics, Hong Kong Baptist University, Kowloon Tong, Kowloon, Hong Kong}
\email{kmuratjan@gmail.com}
\author{Dayoon Park}
\address{Department of Mathematical Sciences, Seoul National University, Seoul 151-747, Republic of Korea}
\email{pdy1016@snu.ac.kr}
\author{Sagnik Saha}
\address{Department of Mathematics, Indian Institute of Science Education and Research, Thiruvananthapuram, Vithura, Kerala 695551, India}
\email{sagniksaha16@iisertvm.ac.in}
\author{Hiu Chun So}
\address{Department of Mathematics, University of Hong Kong, Pokfulam, Hong Kong}
\email{u3538555@connect.hku.hk}
\author{Piyush Varyani}
\address{ Department of Mathematics, Indian Institute of Technology, Roorkee, Roorkee, Uttarakhand 247667, India }
\email{piyushviitr@gmail.com}
\title{Fermat's polygonal number theorem for repeated generalized polygonal numbers}
\thanks{The research presented here was conducted while the second, fourth, sixth, seventh, and eighth authors were undergraduate researcher assistants at the University of Hong Kong and they thank the university for its hospitality. The internship of the second author was additionally supported by the Hong Kong Indian Chamber of Commerce, who he thanks for their generous support. The research of the third author was supported by grants from the Research Grants Council of the Hong Kong SAR, China (project numbers HKU 17316416, 17301317, and 17303618).}
\date{\today}
\begin{abstract}
In this paper, we consider sums of generalized polygonal numbers with repeats, generalizing Fermat's polygonal number theorem which was proven by Cauchy. In particular, we obtain the minimal number of generalized $m$-gonal numbers required to represent every positive integer and we furthermore generalize this result to obtain optimal bounds when many of the generalized $m$-gonal numbers are repeated $r$ times, where $r\in\N$ is fixed. 

\end{abstract}
\maketitle

\begin{section}{Introduction}
Fermat famously conjectured in 1638 that every positive integer may be written as the sum of at most $m$ $m$-gonal numbers; that is, for $P_m(x):=\frac{(m-2)x^2-(m-4)x}{2}$ (the \begin{it}$x$-th $m$-gonal number\end{it}, where $x\in\N_0$ with $\N_0:=\N\cup\{0\}$) there exists an $\bm{x} = (x_1, x_2, \dots, x_m)\in\N_0^m$ such that 
\[
\sum_{j=1}^m P_m(x_j) = n
\]
for every $n\in\N$; we call a Diophantine equation which represents every positive integer \begin{it}universal\end{it}. The $m=4$ case of Fermat's claim was Lagrange's celebrated four squares theorem, proven in 1770, Gauss famously proved the $m=3$ case, sometimes known as the Eureka Theorem, in 1796, and Cauchy finally resolved the general case in 1813 \cite{Cauchy}. Guy \cite{Guy} investigated the question of the optimality of Fermat's polygonal number theorem. That is to say, for which $\ell\in\N$ is the sum 
\begin{equation}\label{eqn:polygonalsum}
\sum_{j=1}^{\ell} P_m(x_j) = n
\end{equation}
universal? More generally, Guy \cite{Guy} considered sums of the type \eqref{eqn:polygonalsum} with more general inputs $x_j\in\Z$ ($P_m(x)$ with $x\in\Z$ is known as a \begin{it}generalized $m$-gonal number\end{it}) and used a simple argument based on the fact that the smallest generalized $m$-gonal number other than $0$ and $1$ is $m-3$ to show that $\ell\geq m-4$ for $m\geq 8$, while Cauchy's theorem implies that the minimal choice satisfies $\ell \leq m$. Comparison of Guy's and Cauchy's theorems hence leaves a small gap between the upper and lower bounds. In this paper, we ask where the true answer lies within this gap in the case of generalized $m$-gonal numbers. For $\bm{a}\in\N^{\ell}$ and $m\geq 3$, consider the sum ($\bm{x}\in\Z^{\ell}$)
\begin{equation}\label{eqn:Pmadef}
P_{m,\bm{a}}(\bm{x}):=\sum_{j=1}^{\ell} a_j P_m(x_j).
\end{equation}
One may think of this as a weighted sum of generalized polygonal numbers or as a sum of generalized polygonal numbers where the first generalized $m$-gonal number is repeated $a_1$ times, the second is repeated $a_2$ times, and so on. Using this second interpretation, we see by Guy's work \cite{Guy} that if $P_{m,\bm{a}}$ is universal, then $\sum_{j=1}^{\ell} a_j\geq m-4$; an upper bound for $\sum_{j=1}^{\ell}a_j$ is not clear, however. We consider the specific case when\[
\bm{a}_{r,\ell_1,\ell}:=(\bm{1},\bm{r})=(1,1,\dots, 1,r,r,\dots, r),
\]
 where $1$ is repeated $\ell_1$ times and $r$ is repeated $\ell_2:=\ell-\ell_1$ times. Let $\ell_m$ denote the minimal $\ell$ for which \eqref{eqn:polygonalsum} is universal when we more generally allow $\bm{x}\in\Z^{\ell}$ and similarly for $r\geq 2$ and $\ell_1\in\N$ denote the optimal minimal choice $\ell=\ell_{m,r,\ell_1}$ for which the sum of generalized $m$-gonal numbers $P_{m,\bm{a}_{r,\ell_1,\ell}}$ defined in \eqref{eqn:Pmadef} is universal. Our main result is the following.
\begin{theorem}\label{thm:lmr}
\noindent

\noindent
\begin{enumerate}[leftmargin=*,align=left,label={\rm(\arabic*)}]
\item 
For $m\notin\{7,9\}$ we have 
\[
\ell_m=\begin{cases} m-4& \text{if }m\geq 10,\\ 
3& \text{if }m\in \{3,5,6\},\\
4&\text{if }m\in\{4,8\}.
\end{cases}
\]

\item 
For $7\leq r< m-3$ we have
\[
\ell_{m,r,r-1}=\ceil{\frac{m-3}{r}}+(r-2).
\]
\item We have 
\begin{align*}
\ell_{m,2,1}&=\floor{\frac{m}{2}} \text{for }m\geq 14,\\
\ell_{m,3,2}&=\begin{cases}m-2 & \text{for }m\geq 10 \text{ with } m \not\equiv 2 \pmod{3},\\ \frac{2m-4}{3} & \text{ for }m\geq 14 \text{ with } m\equiv 2 \pmod{3}, \end{cases}\\
\ell_{m,4,3}&=\ceil{\frac{m-2}{4}}+2\text{ for }m\geq 62,\\
\ell_{m,5,4}&=\ceil{\frac{m-3}{5}}+3 \text{ for }m\geq 78,\\
\ell_{m,6,5}&=\ceil{\frac{m-3}{6}}+4 \text{ for }m\geq 93.
\end{align*}
\end{enumerate}
\end{theorem}
\newpage
\begin{remarks}
\noindent

\noindent
\begin{enumerate}[leftmargin=*,align=left,label={\rm(\arabic*)}]
\item
Using Guy's argument, for $m$ sufficiently large (depending on $r$), if $P_{m,\bm{a}_{r,\ell_1,\ell}}$ is universal, then one must have $\ell_1\geq r-1$ since otherwise the integers from $1$ to $r-1$ cannot all be represented by the form $P_{m,\bm{a}_{r,\ell_1,\ell}}$. Hence Theorem \ref{thm:lmr} (2) is optimal in the $\ell_1$ aspect. The restriction on $r$ is chosen so that we have at least $6$ variables which are not repeated. The cases $2\leq r\leq 6$ hence require more delicate care and lead to weaker results in terms of the dependence on $m$. Indeed, a more careful case-by-case checking shows that one may take $m\geq 27$ for $r=4$, $m\geq 34$ for $r=5$, and $m\geq 40$ for $r=6$, but we have chosen the weaker restrictions on $m$ appearing in Theorem \ref{thm:lmr} (3) in order to present the proof in a more systematic way. These improved lower bounds for $m$ form a theoretical limit on the extent to which the method in this paper may be applied; that is to say, reducing the bound on $m$ beyond the stated bounds $m\geq 14$, $m\geq 14$, $m\geq 27$, $m\geq 34$, and $m\geq 40$ for $r=2$, $r=3$, $r=4$, $r=5$, and $r=6$, respectively, would require a different method than the one presented in this paper (or at least a serious modification that likely depends on the choice of $m$) because we would not have enough variables to apply a crucial lemma that applies to the generic case. Motivated by this, the second, fourth, sixth, seventh, and eighth authors \cite{BKSSV} have relaxed the conditions to $\ell_1=r+4$ in order to guarantee at least $6$ such variables for $r\geq 2$, thereby extending the method in this paper to compute $\ell_{m,r,r+4}$ without any restriction on $r$ or $m$.  

\item
The second restriction in Theorem \ref{thm:lmr} (2) is somewhat artificial. Namely, if $r\geq m-3$, then we have $r-1\geq m-4$ generalized $m$-gonal numbers preceding the $r$-times repeated generalized $m$-gonal numbers, and the original $r-1$ terms are already universal by Theorem \ref{thm:lmr} (1).  
\item
The method used in this paper does not work for the cases $m\in\{7,9\}$ in Theorem \ref{thm:lmr} (1). A certain modification of Lemma \ref{lem:conguniversal} might work for $m=9$, but the $m=7$ case seems to require a different method because the dimension is too small to use a modification of Lemma \ref{lem:conguniversal}. Together with K.-L. Kong, the first and sixth authors are investigating the usage of modular forms techniques to resolve these remaining cases. 

\end{enumerate}
\end{remarks}

The case $r=3$ in Theorem \ref{thm:lmr} (3) is exceptional both because $\ell_{m,3,2}>\ell_m$ and because the dependence on $r$ for $r=3$ is vastly different than the generic dependence on large $r$ in Theorem \ref{thm:lmr} (2). The primary reason for this is the fact that 
\[
P_m(2)=m\equiv m-3=P_m(-1)\pmod{3}.
\]
Because of this, it turns out that either $3m-12$ or $2m-9$ is not represented by $P_{m,\bm{a}_{3,2,\ell}}$ for $\ell<m-2$. Guy exploited a similar property for $m-4$ in order to obtain the lower bound $\ell_m\geq m-4$. 

This special behaviour of the integers $3m-12$ or $2m-9$ brings up an interesting discussion about general forms $P_{m,\bm{a}}$ with arbitrary $\bm{a}\in\N^{\ell}$. Generalizing the diagonal case of the Conway--Schneeberger fifteen theorem, Liu and the third author \cite{KaneLiu} proved that there exists a unique minimal $\gamma_m\in\N$ such that $P_{m,\bm{a}}$ is universal if and only if it represents every $n\leq \gamma_m$. It was shown in \cite{KaneLiu} that $m-4\leq \gamma_m\ll m^{7+\varepsilon}$, and this was improved by the fifth author and Kim \cite{KimPark}, who showed that there exists an absolute constant $c\geq 1$ such that $m-4\leq \gamma_m\leq cm$. It is natural to wonder about the optimal choice of $c$ (perhaps only holding for $m$ sufficiently large). The case $r=3$ leads to the conclusion that $c\geq 3$ unless $m\equiv 2\pmod{3}$, in which case $c\geq 2$. 
\begin{corollary}\label{cor:gammabound}
If $m\geq 14$, then we have 
\[
\gamma_m\geq \begin{cases} 3m-12 &\text{if }m\not\equiv 2\pmod{3},\\ 2m-9&\text{if }m\equiv 2\pmod{3}. \end{cases}
\]
\end{corollary}
\begin{remark}
Using techniques from the arithmetic theory of quadratic forms, the constant $\gamma_m$ has been explicitly computed for some small $m$. In particular, we have $\gamma_3=\gamma_6=8$ by Bosma and the third author \cite{BosmaKane}, $\gamma_4=15$ by the Conway--Schneeberger fifteen theorem \cite{Conway,Bhargava}, $\gamma_5=109$ by Ju \cite{Ju}, and $\gamma_8=60$ by Ju and Oh \cite{JuOh}. In light of the work in \cite{KimPark} and the lower bound in Corollary \ref{cor:gammabound}, it may be interesting to systematically investigate other choices of $\bm{a}$ in order to obtain an improvement on the lower bound for $c$. 
\end{remark}

The paper is organized as follows. In Section \ref{sec:prelim} we give some helpful preliminary information about quadratic forms and quadratic polynomials. In Section \ref{sec:proofs} we prove Theorem \ref{thm:lmr} (1) and Theorem \ref{thm:lmr} (2), giving the stronger version of Fermat's polygonal number theorem in the $r=1$ case and its generalization for large $r$. Finally, in Section \ref{sec:smallr}, we consider small choices of $r>1$, for which a different technique is necessary, and the resulting bound for $\gamma_m$ given in Corollary \ref{cor:gammabound}.

\end{section}
\begin{section}*{Acknowledgements}
The authors thank Min-Joo Jang and Sudhir Pujahari for helpful conversations and the anonymous referee for a careful reading of the paper. 
\end{section}
\begin{section}{Preliminaries}\label{sec:prelim}

The sums of polygonal numbers appearing in \eqref{eqn:Pmadef} are a special case of a natural class of functions known as quadratic polynomials. In order to define these, recall that a homogeneous polynomial $Q$ of degree $2$ is known as a \begin{it}quadratic form\end{it}. If $Q(\bm{x})\in\Z$ whenever $\bm{x}\in\Z^{\ell}$, then we call $Q$ \begin{it}integer-valued\end{it}, and it is moreover known as integral if the associated \begin{it}Gram matrix\end{it} (i.e., the matrix $A$ for which $Q(\bm{x})=\bm{x}^T A\bm{x}$) has integer coefficients ({\bf warning}: in different contexts, authors write $Q(\bm{x})=\frac{1}{2}\bm{x}^T A\bm{x}$, so one needs to be careful about a factor of $2$ whenever comparing in the literature). We call such a quadratic form \begin{it}positive-definite\end{it} if it only attains non-negative values and vanishes if and only if $\bm{x}=\bm{0}$. A \begin{it}totally-positive quadratic polynomial\end{it} is a function of the form 
\[
P(\bm{x})=Q(\bm{x})+\mathcal{L}(\bm{x})+c,
\]
where $Q$ is a positive-definite quadratic form, $\mathcal{L}$ is a linear function defined over $\Z$, and $c$ is a constant, such that $P(\bm{x})\geq 0$ for all $\bm{x}\in\Z^{\ell}$ and $P(\bm{x})=0$ if and only if $\bm{x}=\bm{0}$. We furthermore assume that $P$ attains integer values for $\bm{x}\in\Z^{\ell}$.

For a totally-positive quadratic polynomial $P$, we set
\[
r_P(n):=\#\{\bm{x}\in\Z^{\ell}: P(\bm{x})=n\}.
\]
Note that if $P=Q$ is a quadratic form with associated Gram matrix $A$, then for each matrix $B\in \GL_{\ell}(\Z)$ satisfying 
\[
B^TAB=A
\]
and each $\bm{x}$ such that $Q(\bm{x})=n$, we have 
\[
Q(B\bm{x})=\bm{x}^T B^T A B \bm{x}=\bm{x}^T A \bm{x}=Q(\bm{x})=n.
\]
We call $B$ an \begin{it}automorph\end{it} of $Q$ and set $\omega_Q$ to be the number of automorphs of $Q$. The matrix $B$ is a special case of an \begin{it}isometry\end{it} between two quadratic forms $Q$ and $Q'$; we say that $Q'$ is isometric to $Q$ over a ring $R$ if there exists $B\in \GL_{\ell}(R)$ such that $B^TAB=A'$, where $A$ and $A'$ are the Gram matrices of $Q$ and $Q'$, respectively. The set of isometry classes of a given discriminant is finite

The first check for representations of $n$ by a quadratic polynomial is to test local conditions. Namely, if $P(\bm{x})=n$ is not solvable with $\bm{x}\in \Z_p^{\ell}$ for some prime $p$ (or, equivalently, modulo $p^j$ for some $j$), then clearly $P(\bm{x})=n$ is not solvable with $\bm{x}\in\Z^{\ell}$. An integer is said to be \begin{it}locally represented\end{it} if it is represented over $\Z_p$ for all primes $p$. Minkowski began the study of the local-global principle; this asks for which locally-represented integers $n$ do global representations (representations over the integers in this setting) exist. Siegel defined a natural weighted average 
\[
r_{\operatorname{gen}(Q)}(n):=\frac{1}{\sum_{j=1}^{r} \frac{1}{\omega_{Q_j}}} \sum_{j=1}^r \frac{r_{Q_j}(n)}{\omega_{Q_j}},
\]
where the sum runs over all of the isometry classes of positive-definite quadratic forms $Q_j$ which are isometric to $Q$ over $\Z_p$ for all $p$ (the set of such forms is known as the \begin{it}genus\end{it} of $Q$ and $r$ is known as the \begin{it}class number of $Q$\end{it}). Siegel \cite{Siegel1,Siegel2} and Weil \cite{Weil} then computed so-called local densities (roughly speaking, these ``count'' the number of representations over $\Z_p$ and vanish precisely when no such representations exist) to give an explicit formula for $r_{\operatorname{gen}(Q)}(n)$. We need only the following well-known special form of their results.

\begin{theorem}[Siegel, Weil]\label{thm:SiegelWeil}
We have that $r_{\operatorname{gen}(Q)}(n)>0$ if and only if $n$ is locally represented. Moreover, if the class number of $Q$ is one, then $r_Q(n)>0$ if and only if $n$ is locally represented.
\end{theorem}

The following lemma plays a crucial role in the proof of Theorem \ref{thm:lmr}. 
\begin{lemma}\label{lem:conguniversal}
The sum $\sum_{j=1}^{5}P_m(x_j)$ represents every integer in the set $(m-2)\N_0$.
\end{lemma}
\begin{proof}
Consider $\bm{x}\in\Z^5$ in the hyperplane $\sum_{j=1}^{5}x_j = 0$. For $\bm{x}$ in this hyperplane, we have 
\begin{align*}
\sum_{j=1}^{5}P_m(x_j)&=\frac{m-2}{2}\sum_{j=1}^{5}x_j^2-\frac{m-4}{2}\sum_{j=1}^{5}x_j\\
&=\frac{m-2}{2}\left(\sum_{j=1}^{4}x_j^2 +\left(-x_1-x_2-x_3-x_4\right)^2 \right)\\
&=(m-2)\sum_{1\leq i\leq j\leq 4} x_ix_j.
\end{align*}
The quadratic form $\sum_{1\leq i\leq j\leq 4} x_ix_j$ has class number one and represents every integer locally, and is hence universal by Theorem \ref{thm:SiegelWeil} (alternatively, one may simply use the $290$-theorem of Bhargava and Hanke \cite{BhargavaHanke} and verify that every integer up to $290$ is represented by this quadratic form, and thus the form is universal).

\end{proof}
\end{section}

\begin{section}{The extension of Fermat's polygonal number theorem for \texorpdfstring{$r=1$}{r=1} and large \texorpdfstring{$r$}{r}}\label{sec:proofs}
In this section, we prove parts (1) and (2) of Theorem \ref{thm:lmr}, giving the generalization of Fermat's polygonal number theorem answering Guy's question and covering the generic case for $r\geq 7$.

\subsection{The case \texorpdfstring{$r=1$}{r=1}}
We next make use of Lemma \ref{lem:conguniversal} in order to prove Theorem \ref{thm:lmr} (1).
\begin{proof}[Proof of Theorem \ref{thm:lmr} (1)]
The case $m=3$ was proven by Gauss, the case $m=4$ was proven by Lagrange, and Guy \cite{Guy} uses Legendre's classification of the integers which are sums of three squares to resolve the $m=5$ case. Guy also points out that the set of generalized hexagonal numbers is precisely the set of triangular numbers and hence the $m=6$ case follows from the $m=3$ case. The $m=8$ case is proven by Sun in \cite[Theorem 1.1]{Sun}. 
\vspace{.1in}

Now assume that $m\geq 10$. Since we know that $\ell_{m}\geq m-4$ by Guy's work in \cite{Guy}, it suffices to prove that for every integer can be written as the sum of $m-4$ generalized $m$-gonal numbers. Let $n\in\N$ be given and write it as 
\[
n=(m-2)k_1+k_2
\]
with $0\leq k_2\leq m-3$. By Lemma \ref{lem:conguniversal}, every multiple of $m-2$ may be written as the sum of $5$ generalized $m$-gonal numbers. Hence if $k_2$ may be written as a sum of $m-9$ generalized $m$-gonal numbers, then we may choose $\bm{x}\in\Z^{m-4}$ for which $\sum_{j=1}^{5}P_m(x_j)=(m-2)k_1$ and $\sum_{j=6}^{m-4} P_{m}(x_j)=k_2$, yielding the claim. This is possible for $0\leq k_2\leq m-9$ and $k_2=m-3$ (because $P_m(-1)=m-3$). 

It remains to consider the cases $m-8\leq k_2\leq m-4$. We thus write $k_2=m-2-k$ with $2\leq k\leq 6$. For $k_1\geq k-1$ we may write 
\begin{align*}
n&=(m-2)k_1+k_2=(m-2)(k_1-k)+k(m-2)+(m-2-k)\\
&=(m-2)(k_1-k+1)+k(m-3),
\end{align*}
from which we conclude that $n$ may be written as the sum of $5+k$ generalized $m$-gonal numbers (again using Lemma \ref{lem:conguniversal}). If $5+k\leq m-4$ (i.e., $m\geq k+9$ which is automatically true for $m\geq 15$), then we see that $n$ is represented as long as $k_1\geq k-1$. On the other hand, if $m<k+9$, then we note that 
\[
P_m(-k)=(m-2)\frac{k^2+k}{2}-k \implies k_2=m-2-k= P_m(-k)-(m-2)\left(\frac{k^2+k}{2}-1\right)
\]
and write
\[
n=(m-2)k_1+k_2=(m-2)\left(k_1+1-\frac{k^2+k}{2}\right)+P_m(-k).
\]
Using Lemma \ref{lem:conguniversal}, $n$ may hence be written as the sum of $6\leq m-4$ generalized $m$-gonal numbers as long as $k_1\geq \frac{k^2+k}{2}-1$. 

It remains to show that $n$ may be represented in the finitely many cases $0\leq k_1<k-1$ (resp. $0\leq k_1< \frac{k^2+k}{2}-1$) when $m\geq k+9$ (resp. $10\leq m<k+9$), with $2\leq k\leq 6$. First suppose that $m\geq k+9$. For $0\leq k_1< k-1$ we write 
\[
(m-2)k_1+k_2= mk_1 + m-2-2k_1-k.
\]
For $m\geq 2+2k_1+k$ (in particular, since $k_1\leq k-2$ and $k\leq 6$, this holds for $m\geq 16$) we see that $n$ may be represented by using $k_1$ choices of $m$ and $m-2-2k_1-k$ choices of $1$, for which we need (using that $0\leq k_1<k-1$ and $k\geq 2$) 
\[
k_1+m-2-2k_1-k=m-2-k_1-k\leq m-2-k\leq m-4
\]
variables. The result follows except for the case $m=15$, $k_1=4$, and $k=6$, for which one may check by hand that $59=12+3\cdot 15 + 2\cdot 1$ may be written as the sum of 6 generalized $15$-gonal numbers. 

In the remaining cases, we have $10\leq m< k+9\leq 15$ and $0\leq k_1< \frac{k^2+k}{2}-1\leq 20$. There are hence only a finite number of $n$ which need to be checked, and this may be done by hand.\qedhere
\end{proof}
\begin{remark}
After reducing the proof to a check of finitely-many cases, we simply check the remaining cases by hand for $10\leq m<k+9\leq 15$ and $0\leq k_1<20$. One may instead drop the restriction $10\leq m<15$ (leaving $m$ arbitrary as a variable) and use (the following list is complete for $P_m(x)\leq 21m-35$ because the sequence $(P_m(0),P_m(1),P_m(-1),P_m(2),P_m(-2),\dots)$ is increasing for $m>3$)
\begin{multline}\label{eqn:Pmvals}
\{ P_{m}(x): x\in\Z\}=\{0,1,m-3,m,3m-8,3m-3, 6m-15, 6m-8,\\
 10m-24,10m-15, 15m-35, 15m-24, 21m-48,21m-35, \dots\}
\end{multline}
 to systematically write $(m-2)k_1+(m-2-k)$ (thinking of this as a polynomial in $m$) as a linear combination of the polynomials occurring in \eqref{eqn:Pmvals} for each choice of $0\leq k_1<20$.
\end{remark}

\begin{subsection}{Inequalities for large \texorpdfstring{$r$}{r}}

For $n\in\N$ and $r<m-3$, we write $n\in\N$ in the form 
\begin{equation}\label{eqn:splitn}
    n = (m-2)k_1 + rk_2 + k_3,
\end{equation}
 where $0 \leq k_2 \leq \floor{\frac{m-3}{r}}$ and $-5 \leq k_3 \leq r-6$. In order to obtain an upper bound, we need the following extension of Lemma \ref{lem:conguniversal}. 
\begin{lemma}\label{lem:splituniversal}
Suppose that $7\leq r<m-3$. For $k_1\in\N_0$ and $-5\leq k_3\leq r-6$, the integer $k_1(m-2)+k_3\in\N_0$ is represented by the sum of at most $r-1$ generalized $m$-gonal numbers unless $-5\leq k_3\leq -1$ and $k_1\leq |k_3|-1$.
\end{lemma}
\begin{proof}
Using Lemma \ref{lem:conguniversal}, we may represent $(m-2)k_1$ with the first $5$ variables. If $0\leq k_3\leq r-6$, then we may represent $k_3$ by taking $P_{m}(x_j)\in\{0,1\}$ for the remaining $r-6$ generalized $m$-gonal numbers.

Now suppose that $-5\leq k_3\leq -1$. We note that 
\[
P_m(x)=(m-2)P_3(-x)+x.
\]
Hence in particular we have 
\[
P_m\left(k_3\right)=(m-2)P_3(-k_3) +k_3. 
\]
We may therefore rewrite \eqref{eqn:splitn} as 
\[
n=(m-2)\left(k_1-P_3(-k_3)\right)+rk_2 +  P_m\left(k_3\right),
\]
and for $k_1\geq P_3(-k_3)$ we conclude that $n-rk_2$ may be represented with the first $6\leq r-1$ variables.

It remains to show that $n$ is represented for the cases $k_1< P_3(-k_3)$ and $-5\leq k_3\leq -1$. For all of these cases other than the exceptional cases
\[
(m-2)k_1+k_3\in \{m-7,m-6,m-5,m-4,2m-9,2m-8,2m-7,3m-11,3m-10,4m-13\}
\]
we may use \eqref{eqn:Pmvals} (thinking of $(m-2)k_1+k_3\in\Z[m]$ as a polynomial in $m$) to find a representation 
\[
(m-2)k_1+k_3=\sum_{j=1}^5P_m(x_j)
\]
in $5< r-1$ variables.  We encode the representations in a graph in the following manner. Write the numbers in rows and columns, where the $(A,B)$ position corresponds to $A(m-2)+B$. If we have a representation 
\[
A(m-2)+B=\sum_{j=1}^{d} P_m(x_j)
\]
in $d$ variables and $A(m-2)+B+P_m(x)=C(m-2)+D$, then in the $(C,D)$ location of the graph we write $(C,D)_{d+1}^{x}$ to indicate that we have a representation of $C(m-2)+D$ in $d+1$ variables where we take $x_{d+1}=x$. One may then reconstruct the representation of $C(m-2)+D$ by recursively working backwards through the graph; for example, if we have $(C,D)_{d+1}^{-1}$, then we obtain the representation by looking at $(C-1,D+1)_{d}^{*}$ and continuing recursively until we have $d=1$. To summarize, one traverses backwards through the graph as follows:
\begin{align*}
(C,D)_{d+1}^{-1}&\to (C-1,D+1)_{d}^{*},\\
(C,D)_{d+1}^{1}&\to (C,D-1)_{d}^{*},\\
(C,D)_{d+1}^{2}&\to (C-1,D-2)_{d}^{*},\\
(C,D)_{d+1}^{-2}&\to (C-3,D+2)_{d}^{*},\\
(C,D)_{d+1}^{-4}&\to (C-10,D+4)_{d}^{*}.
\end{align*}
This yields the following graph encoding the representations (we add the unnecessary entries $(C,-1)$ in order to include the representations of some integers in the $(-5)$th column)
\[
\begin{array}{lllll}
(1,-5)&(1,-4)&(1,-3)&(1,-2)&(1,-1)_{1}^{-1}\\
(2,-5)&(2,-4)&(2,-3)&(2,-2)_2^{-1}&(2,-1)_{3}^{1}\\
(3,-5)&(3,-4)&(3,-3)_3^{-1}&(3,-2)_1^{-2}&(3,-1)_{2}^{1}\\
(4,-5)&(4,-4)_4^{-1}&(4,-3)_2^{-2}&&(4,-1)_{4}^{2}\\
(5,-5)_5^{-1}&(5,-4)_3^{-1}&(5,-3)_4^{-2}&&\\
(6,-5)_4^{-1}&(6,-4)_2^{-2}&(6,-3)_1^{-3}&&\\
(7,-5)_3^{-1}&(7,-4)_2^{-1}&&&\\
(8,-5)_3^{-1}&(8,-4)_4^{1}&&&\\
(9,-5)_2^{-2}&(9,-4)_3^{1}&&&\\
(10,-5)_4^{-1}&(10,-4)_1^{-4}&&&\\
(11,-5)_2^{-1}&&&&\\
(12,-5)_4^{-4}&&&&\\
(13,-5)_3^{-4}&&&&\\
(14,-5)_5^{-4}&&&&\\
(15,-5)_{1}^{-5}&&&&
\end{array}
\]
\end{proof}
For the exceptional cases $-5\leq k_3\leq -1$ and $k_1\leq |k_3|-1$, we use the following lemma.
\begin{lemma}\label{lem:exceptional}
If $7\leq r<m-3$, $k_2\geq 1$, and $k_1(m-2)+k_3$ satisfies $-5\leq k_3\leq -1$ and $0\leq k_1\leq |k_3|-1$, then 
\[
n=k_1(m-2)+rk_2+k_3
\]
may be represented by 
\[
\sum_{j=1}^{r-1}P_m(x_j)+r\sum_{j=r}^{r+k_2-2}P_m(x_j).
\]
In particular, we may take $\ell_2\geq k_2-1$. 
\end{lemma}
\begin{proof}
\noindent

\noindent
For some $0\leq j\leq k_1$ we have
\begin{multline*}
n=(m-2)k_1+(k_2-1)r+(r+k_3)=(k_1-j) m+j(m-3) +(r+k_3-2k_1+3j) +(k_2-1)r\\
=(k_1-j) P_m(2)+jP_m(-1) +(r+k_3-2k_1+3j)P_m(1) +(k_2-1)r.
\end{multline*}
If $r+k_3-2k_1+3j\geq 0$, then 
\[
(k_1-j)P_m(2)+jP_m(-1)+(r+k_3-2k_1+3j)P_m(1)
\]
is the sum of $r+k_3-k_1+3j$ generalized $m$-gonal numbers. Hence if the system of equations
\begin{align*}
r+k_3-k_1+3j&\leq r-1,\\
r+k_3-2k_1+3j&\geq 0
\end{align*}
holds, then we are done. If $|k_3|+2k_1\leq 6$, then since $r-1\geq 6$ we may take $j=0$. For $7\leq |k_3|+2k_1\leq 10$ the inequality $k_1\leq |k_3|-1$ implies that $|k_3|+k_1> 3$, and hence we may take $j=1$ in that case. Finally, if $11\leq |k_3|+2k_1\leq 13$, then $|k_3|+k_1>6$, so we may take $j=2$ in this case. 

\end{proof}

We are now ready to obtain an upper bound for $\ell_{m,r,r-1}$ for large $r$.
\begin{proposition}\label{prop:upperbound}
If $7\leq r< m-3$, then we have $\ell_{m,r,r-1} \leq \floor{\frac{m-3}{r}}+(r-1)$.
\end{proposition}
\begin{proof}
The claim is equivalent to proving that $P_{m,\bm{a}_{r,r-1,\ell}}$ is universal for $\ell_2=\floor{\frac{m-3}{r}}$. 

Since $k_2\leq \ell_2$, we may represent $rk_2$ with the $r$-times repeated variables all having $x_j\in\{0,1\}$ (i.e., $P_m(x_j)\in\{0,1\}$), and Lemma \ref{lem:splituniversal} implies that $n-rk_2=(m-2)k_1+k_3$ may be represented by the initial $r-1$ variables unless $-5\leq k_3\leq -1$ and $k_1\leq |k_3|-1$. 

We finally deal with the cases $n=k_1(m-2)+k_2r+k_3$ with $0\leq k_1\leq |k_3|-1$ and $-5\leq k_3\leq -1$. If $k_2\geq 1$, then Lemma \ref{lem:exceptional} implies that $n$ is represented. 

It remains to resolve the $k_2=0$ case for $-5\leq k_3\leq -1$ and $1\leq k_1\leq |k_3|-1$.  In other words, we need to check the representations of the 10 integers $k_1(m-2)+k_3\in \{m-7,m-6,m-5,m-4,2m-9,2m-8,2m-7,3m-11,3m-10,4m-13\}$. We write 
\[
m=rs+t
\]
for some $0\leq t\leq r-1$. This gives
\[
n=k_1(m-2)+k_3= \left(k_1-1\right) P_m(2) +rs+t+k_3-2k_1.
\]
If $t+k_3-2k_1\geq 0$, then we are done because $k_1-1+(t+k_3-2k_1)<t\leq r-1$. If $t\leq |k_3|+k_1$, then we may write (note that $s\geq 1$ because otherwise $m=t\leq r-1$, which contradicts the assumption that $m-3> r\geq 7$)
\begin{equation}\label{eqn:nsplitsmall}
n=\left(k_1-j-1\right)P_m(2)+jP_m(-1) + (s-1)r + t+k_3-2k_1+3j+r.
\end{equation}
Noting that $s-1\leq \left\lfloor\frac{m-3}{r}\right\rfloor$, we are done as long as 
\begin{align*}
t+k_3-2k_1+r+3j&\geq 0,\\
k_1-j-1&\geq 0,\\
t+k_3-k_1+3j&\leq 0,
\end{align*}
with the last inequality coming from the fact that $n-(s-1)r$ must be represented by $k_1-j-1+j + t+k_3-2k_1+3j+r\leq r-1$ generalized $m$-gonal numbers.  As in the proof of Lemma \ref{lem:exceptional} this holds for some $j\in \{0,1\}$. 

We finally deal with the case $|k_3|+k_1< t<|k_3|+2k_1$. Since $t>|k_3|+k_1\geq 3$ in this case, we have $m-3=rs+t-3\geq rs$ and hence $s=\left\lfloor\frac{m-3}{r}\right\rfloor$. In this case, we rewrite 
\[
n=k_1P_{m}(-1)+k_1+k_3=\left(k_1-1-j\right)P_{m}(-1)+jP_m(2) + rs +t-3-3j+k_1+k_3.
\]
Writing $t=k_1+|k_3|+t'$ with $1\leq t'<k_1$, we are done as long as 
\begin{align*}
t-3+k_1+k_3-3j&\geq 0,&2k_1-3-3j+t'&\geq 0,\\
k_1-1-j&\geq 0,\qquad \qquad\qquad \Leftrightarrow&k_1-1-j&\geq 0,\\
t+2k_1+k_3-3j-4 &\leq r-1,&3k_1 + t'-3j-4&\leq r-1,
 \end{align*}
with the last inequality coming from the fact that we must write $n-rs$ as the sum of at most $k_1-1+t-3-3j+k_1+k_3$ generalized $m$-gonal numbers.  Setting $\delta:=1$ if $t'=k_1-1$ and $\delta=0$ otherwise, we claim that $j=t'-\delta$ satisfies the above system of inequalities. Since $j\leq t'\leq k_1-1$, the second inequality automatically holds. The first inequality 
\[
2k_1-3-3j+t'=2\left(k_1-t'\right) -3 +3\delta\geq 0
\]
holds because $k_1-t'\geq 1$, with $k_1-t'=1$ if and only if $\delta=1$. The third inequality becomes 
\[
3k_1-2t'+3\delta -4\leq r-1. 
\]
Note that since $t'\geq 1$, we have $-2t'+3\delta\leq -2$ unless $t'=k_1-1\leq 2$. In the exceptional case $k_1\leq 3$ and $t'=k_1-1$, we have 
\[
3k_1-2t'+3\delta-4= k_1+1\leq 4<r-1
\]
Otherwise, we have $-2t'+3\delta \leq -2$, $k_1\leq 4$, and $r\geq 7$, so we find that
\[
3k_1-2t'+3\delta -4\leq 3\left(k_1-2\right)\leq 6\leq r-1,
\]
and the claim follows.
\end{proof}
We next use Guy's argument to obtain a lower bound for $\ell_{m,r,r-1}$. 
\begin{proposition}\label{prop:lowerbound}
We have $\ell_{m,r,r-1} \geq \ceil{\frac{m-3}{r}}+(r-2)$.
\end{proposition}
\begin{proof}
Following Guy \cite{Guy}, if $P_{m,\bm{a}}$ is universal, then it must necessarily represent $m-4$. We write 
\begin{equation}\label{eqn:m-4sum}
\begin{split}
    m-4 = \sum_{j=1}^{r-1}P_m(x_j) + r\sum_{j=r}^{\ell}P_m(x_j)
\end{split}
\end{equation}
By \eqref{eqn:Pmvals}, we have $P_m(0) = 0$,  $P_m(1) = 1$, and $P_m(x) >m-4$ for $x\notin\{0,1\}$, so any representation of $m-4$ may only contain $x_j\in\{0,1\}$. Thus \eqref{eqn:m-4sum} yields the inequality
\[
m-4\leq r-1+r(\ell-r+1),
\]
from which we conclude that 
\begin{equation}\label{eqn:lmreval}
\begin{split}
\ell_{m,r,r-1}\geq \ell \geq \frac{m-3}{r}+r-2.
\end{split}
\end{equation}
This yields the claim.
\end{proof}

We are now ready to Prove Theorem \ref{thm:lmr} (2). 
\begin{proof}[Proof of Theorem \ref{thm:lmr} (2)]
The upper bound in Proposition \ref{prop:upperbound} and the lower bound in Proposition \ref{prop:lowerbound} match unless $r\mid m-3$. In the remaining case, we write $m-3=rs$ and note that the claim is equivalent to proving that $P_{m,\bm{a}_{r,r-1,\ell}}$ is universal for $\ell_2=s-1$. Recall the presentation \eqref{eqn:splitn} of $n$. By Lemma \ref{lem:splituniversal}, if $0\leq k_2\leq s-1$, then $n-rk_2$ may be represented by the initial $r-1$ generalized $m$-gonal numbers and we only require $k_2\leq\ell_2$ variables to represent $rk_2$, unless $k_3<0$ and $k_1\leq |k_3|-1$. For $k_3<0$ and $k_1\leq |k_3|-1$, we use Lemma \ref{lem:exceptional} to see that $n$ is represented with $\ell_2=s-1$ unless $k_2=0$, while for $k_2=0$ we use the splitting \eqref{eqn:nsplitsmall} (with $t=3$) with $j\in\{0,1\}$ to obtain a representation. 

For $k_2=s$, rewrite 
\[
n=k_1(m-2)+m-3+k_3=(k_1+1)(m-2)+k_3+1. 
\]
Again using Lemma \ref{lem:splituniversal}, we see that $n$ is represented by $r-1$ generalized $m$-gonal numbers unless ($-5\leq k_3\leq -2$ and $k_1+1\leq |k_3|-2$) or $k_3=r-6$. We use \eqref{eqn:nsplitsmall} in the case $-5\leq k_3\leq -2$. In the case of $k_3=r-6$, we then rewrite 
\[
n=(k_1+1)(m-2)+r -5.
\]
In this case, Lemma \ref{lem:splituniversal} implies that $n$ is represented unless $k_1\leq 3$, while Lemma \ref{lem:exceptional} with $k_2=1$ yields the claim for $k_1\leq 3$. 

\end{proof}

\end{subsection}

\end{section}

\begin{section}{Small choices of \texorpdfstring{$r$}{r}}\label{sec:smallr}
In this section, we consider cases for small $r$. 
\begin{proof}[Proof of Theorem \ref{thm:lmr} (3)]
We first assume that $r=2$ and $m\geq 14$. Note that since $m-3\not\equiv m-2\pmod{2}$, any representation of $m-2$ by $P_{m,\bm{a}_{2,1,\ell}}$ must have $\ell_2\geq \floor{\frac{m}{2}}-1$, or in other words $\ell\geq \floor{\frac{m}{2}}$, yielding the lower bound $\ell_{m,2,1}\geq \floor{\frac{m}{2}}$.

It remains to show that the form $P_{m,\bm{a}_{2,1,\ell}}$ with $\ell:=\floor{\frac{m}{2}}$ is indeed universal. We present $n$ in the form
\begin{equation}\label{eqn:splitnr=2}
n=2(m-2)k_1+2k_2+k_3.
\end{equation}
with $-5\leq k_2\leq m-8$ and $k_3\in\{0,1\}$ (so $2k_2+k_3$ precisely attains every residue modulo $2(m-2)$ once). By Lemma \ref{lem:conguniversal}, we may represent $2(m-2)k_1$ as a sum of the type $2\sum_{j=2}^6 P_m(x_j)$. If $0\leq k_2\leq \ell-6$, then we conclude that $n$ may be represented by $P_{m,\bm{a}_{2,1,\ell}}$.

We next consider $\ell-5 \leq k_2\leq m-8$. Choosing $j\in\{-1,2\}$ such that $P_m(j)\equiv k_3\pmod{2}$, we may rewrite \eqref{eqn:splitnr=2} as 
\[
n=2(m-2)k_1+2\left(k_2+\frac{k_3-P_m(j)}{2}\right)+P_m(j).
\]
Setting $k:=k_2+\frac{k_3-P_m(j)}{2}$, the inequalities for $k_2$ imply that
\[
\ell-5-\frac{m}{2} \leq  k\leq m-8- \frac{m-4}{2}=\frac{m}{2}-6.
\]
Since $k$ is an integer, we conclude that $-5\leq k\leq \ell-6$. For $0\leq k\leq \ell-6$ we are done by Lemma \ref{lem:conguniversal}, and for $-5\leq k\leq -1$ we use Lemma \ref{lem:splituniversal} (choosing $r=7$) to conclude that $(m-2)k_1+k$ may be written as the sum of at most $6\leq \ell-1$ (because $m\geq 14$ we have $\ell\geq 7$) generalized $m$-gonal numbers unless $k_1\leq |k|-1$.

We next consider the cases $-5\leq k_2\leq -1$. In this case we write 
\[
n= 2\left(k_1(m-2)+k_2\right) +k_3
\]
and, since $-5\leq k_2\leq -1$, Lemma \ref{lem:splituniversal} (choosing $r=7$) implies that $k_1(m-2)+k_2$ may be written as the sum of at most $6\leq \ell-1$ (because $m\geq 14$ we have $\ell\geq 7$) generalized $m$-gonal numbers unless $k_1\leq |k_2|-1$. 

Note that $k_3\in \{P_m(0),P_m(1)\}$. It remains to show that for each $-5\leq k'\leq -1$ and $0\leq k_1\leq |k'|-1$ and $-1\leq j'\leq 2$ there exists a representation 
\begin{equation}\label{eqn:r=2special}
P_m(x_1)+2\sum_{j=2}^{\ell}P_m(x_j)= n=2\left((m-2)k_1+k'\right)+P_m(j').
\end{equation}
For each $n$ of the form \eqref{eqn:r=2special}, we claim that we may choose $x_1,\dots, x_{d+1}$ (with $d\in\N_0$) so that $P_m(x_1)\equiv n\pmod{2}$ and 
\begin{equation}\label{eqn:r=2bnd}
0\leq n-P_m(x_1)-2\sum_{j=2}^{d+1} P_m(x_j)\leq 2\ell_2-2d.
\end{equation}
Note that if we may choose $x_j$ in this way, then since 
\[
n':= n-P_m(x_1)-2\sum_{j=2}^{d+1} P_m(x_j)
\]
 is even and less than $2(\ell_2-d)$, we may write $n'$ as a sum of at most $\ell_2-d$ twos, giving a representation of $n$ with $\ell$ variables. It remains to choose the first $d$ of the $x_j$s appropriately. 

We collect the choices of the set $X_n$ such that $x_1\in X_n$, $x_2,\dots, x_{d+1}$ and the corresponding bounds on $n'$ in Table \ref{tab:r=2special}. The bounds on $n'$ are proven, for example, for $2m-14\leq n\leq 2m-7$ and $x_1\in X_n=\{-1,2\}$ by writing (using $m\geq 14$)
\[
0\leq m-14=2m-14 -m \leq n-P_m(x_1)\leq 2m-7-(m-3)=m-4.
\]
Recalling that $\ell_2=\floor{\frac{m}{2}}-1$, we have $m-4<m-3\leq 2\ell_2$ (in general, it suffices to show that $m-14\leq n'\leq 2\ell_2-2d$), and we see that \eqref{eqn:r=2bnd} holds. 
\begin{center}
\begin{table}[th]\caption{Individual case checking for $r=2$\label{tab:r=2special}}
\begin{tabular}{|c|c|c|c|l|}
\hline
Interval with $n$& $x_1\in X_n$ & $d$ & $x_2,\dots,x_{d+1}$ & Bounds on $n'$\\
\hline
\hline
$m-13\leq n\leq m-2$&$\{0,1\}$& $0$&& $m-14\leq n'\leq m-2$\\
\hline
$2m-14\leq n\leq 2m-7$&$\{-1,2\}$& $0$&& $m-14\leq n'\leq m-4$\\
\hline
$3m-17\leq n\leq 3m-10$&$\{0,1\}$& $1$&$-1$& $m-12\leq n'\leq m-4$\\
\hline
$4m-18\leq n\leq 4m-13$&$\{-1,2\}$& $1$&$-1$& $m-12\leq n'\leq m-4$\\
\hline
$5m-21\leq n\leq 5m-18$&$\{0,1\}$& $2$&$2,2$& $m-10\leq n'\leq m-6$\\
\hline
$5m-17\leq n\leq 5m-14$&$\{0,1\}$& $2$&$-1,2$& $m-12\leq n'\leq m-8$\\
\hline
$6m-22\leq n\leq 6m-19$&$\{-2,3\}$& $1$&$-1$& $m-13\leq n'\leq m-5$\\
\hline
$7m-25\leq n\leq 7m-20$&$\{0,1\}$& $3$&$-1,-1,2$& $m-14\leq n'\leq m-8$\\
\hline
$n=8m-26$&$\{-2,3\}$& $2$&$-1,-1$& $m-11\leq n'\leq m-6$\\
\hline
$n=8m-25$, $m$ odd &$\{-2\}$& $2$&$-1,2$& $n'=m-11$\\
\hline
$n=8m-25$, $m$ even &$\{3\}$& $2$&$-1,-1$& $n'=m-10$\\
\hline
$9m-29\leq n\leq 9m-28$&$\{0,1\}$& $4$&$-1,-1,-1,2$& $m-12\leq n'\leq m-10$\\
\hline
$n= 9m-27$&$\{-1\}$& $4$&$-1,-1,-1,-1$& $n'=0$\\
\hline
$n= 9m-26$&$\{-2\}$& $3$&$-1,-1,-1$& $n'=0$\\
\hline
\end{tabular}
\end{table}
\end{center}
We next consider the $r=3$ case. In this case, first note that Theorem \ref{thm:lmr} (1) implies that $3\sum_{j=3}^{m-2}P_m(x_j)$ represents every element of $3\N_0$. Taking $P_m(x_1),P_m(x_2)\in\{0,1\}$, we get a representation of every positive integer.

We note that any representation of $3m-10$ must have $P_m(x_1)\equiv P_m(x_2)\equiv 1\pmod{3}$ and hence for $m\equiv 0\pmod{3}$ the congruence $m\equiv m-3\equiv 0\pmod{3}$ implies that $x_1=x_2=1$ so that 
\[
3m-12=3\sum_{j=3}^{\ell} P_m(x_j).
\]
Dividing by $3$ and using Guy's argument \cite{Guy} again, this implies that $\ell-2\geq m-4$, or in other words $\ell\geq m-2$. 

For $m\equiv 1\pmod{3}$, we take $n=3m-12$ and similarly note that any representation of 
\[
3m-12=P_m(x_1)+P_m(x_2)+3\sum_{j=3}^{\ell} P_m(x_j)
\]
must have $x_1=x_2=0$, again implying that $\ell\geq m-2$. 

For $m\equiv 2\pmod{3}$, we take $n=2m-9$ and similarly note that, since $2m-9<2m-6$, any representation of 
\[
2m-9=P_m(x_1)+P_m(x_2)+3\sum_{j=3}^{\ell} P_m(x_j)
\]
must have $x_1=0$, $x_1=1$, and $P_m(x_j)\leq 1$, from which we conclude that $\ell\geq \frac{2m-4}{3}$. 

We note that any representation of $n<3(m-2)$ must satisfy 
\[
P_m(x_1)+P_m(x_2)\in \begin{cases}
 \{0,m-2,m+1 \} &\text{if }n\equiv 0\pmod{3},\\
 \{1,2m-6,2m-3,2m,3m-8\} &\text{if }n\equiv 1\pmod{3},\\
 \{2,m-3,m,3m-7\} &\text{if }n\equiv 2\pmod{3}.
\end{cases}
\]
From this we can conclude that $P_3(x_1)+P_3(x_2)+3\sum_{j=3}^{\frac{2m-4}{3}-5} P_m(x_j)$ represents every positive integer up to $3(m-2)$ except integers in
\begin{multline*}
K_3:=\{2m-21,2m-18,2m-15,2m-12,2m-9, 3m-22, 3m-21,3m-19,\\
3m-18,3m-16,3m-15,3m-13,3m-12, 3m-10\}.
\end{multline*}
By Lemma \ref{lem:conguniversal}, $n=3(m-2)k_1+k_3$ with $k_1\in \N_0$, $0\le k_3<3(m-2)$ and $k_3 \not\in K_3$ is represented by $P_3(x_1)+P_3(x_2)+3\sum_{j=3}^{\frac{2m-4}{3}} P_m(x_j)$.

On the other hand, for each $k_3\in K_3$ we write 
\[
k_3=3(j_0(m-2)-k)+\alpha
\]
 with $j_0\in\{0,1\}$ and $\alpha\in \{0,2,2m-6\}$ and we see that $k\leq 5$ except for $k_3=3m-22$, in which case $k=6$. Thus for every $3m-22\neq k_3\in K_3$, Lemma \ref{lem:splituniversal} implies that $k_3-\alpha+3j(m-2)$ is represented as $3$ times the sum of at most $6$ generalized $m$-gonal numbers for $j\geq k-j_0$. Using \eqref{eqn:Pmvals} we may check the smaller choices of $j$ directly. For the remaining case $k_3=3m-22$, we write 
\begin{multline*}
k_3+3j(m-2)=3((j+1)(m-2)-6)+2\\
 = 3(j(m-2)-5+m-3)+2 = 3(j(m-2)-5+P_m(-1))+2.
\end{multline*}
Thus, using Lemma \ref{lem:splituniversal} to represent $j(m-2)-5$, for every $j\geq 5$ we may represent $k_3+3j(m-2)$ as long as $\ell\geq 9$. There remain finitely many choices of $j$ for each $k_3\in K_3$ and we check these as in the $r=2$ case. 
Now suppose that $4\leq r\leq 6$. We first obtain lower bounds for $\ell_{m,r,r-1}$ by using Guy's argument \cite{Guy} for the exceptional choices of $n$ in Table \ref{tab:badnbigr}.
\begin{center}
\begin{table}[th]\caption{Exceptional $n$ for $4\leq r\leq 6$\label{tab:badnbigr}}
\begin{tabular}{c|l|c|c}
$r$ & $m\pmod{r}$& $n$&lower bound\\
&&& for $\ell_{m,r,r-1}$\\
\hline
\hline
$4$&$m\not\equiv 3\pmod{4}$&$m-4$&$\ceil{\frac{m-3}{4}}+2$\\
$4$&$m\equiv 3\pmod{4}$&$2m-4$&$\ceil{\frac{m-3}{4}}+3$ \\
$5$&all &$m-4$&$\ceil{\frac{m-3}{5}}+3$\\
$6$&all &$m-4$&$\ceil{\frac{m-3}{6}}+4$
\end{tabular}
\end{table}
\end{center}

We define sets $S_r$ by 
\begin{align*}
S_4&:=\{0,1,2,3\}\cup\{ m+j: -3\leq j\leq 2\}\cup \{2m-6,2m-5,2m-3,2m-2,2m,2m+1\}\\
&\qquad \cup\{ 3m+j: -9\leq j\leq 6\text{ or }-3\leq j\leq 0\}\cup\{4m-11,4m-10\},\\
S_5&:=\{0,1,2,3,4\}\cup\{m+j: -3\leq j\leq 3\}\cup \{ 2m+j: -6\leq j\leq 2\}\\
&\qquad \cup \{ 3m+j: -9\leq j\leq -5\text{ or }-3\leq j\leq 1\}\cup \{4m+j: -12\leq j\leq 0\},
\\
S_6&:=\{0,1,2,3,4,5\}\cup\{m+j: -3\leq j\leq 4\}\cup \{2m+j: -6\leq j\leq 3\}\\
&\qquad \cup\{3m+j:-9\leq j\leq 2\}\cup \{4m+j: -12\leq j\leq 1\}\\
&\qquad \cup\{5m+j: -12\leq j\leq 0\}\cup\{6m+j:-15\leq j\leq -13\}.
\end{align*}
The sets $S_r$ are precisely the integers less than $r(m-2)$ which are represented by $\sum_{j=1}^{r-1}P_m(x_j)$.

For each $n\in \N$, we choose $s\in S_r$ and $k_1,k_2\in\N_0$ with $k_2$ minimal such that  
\[
n=s+r(m-2)k_1+rk_2.
\]
By Lemma \ref{lem:conguniversal}, we obtain a representation of $n$ with $r-1+5+k_2$ variables, taking $x_j=1$ for the last $k_2$ variables. If $k_2\leq \ell_2-5$ with $\ell=\ell_{m,r,r-1}$ as given in the statement of the theorem, then $n$ may be represented. We check in Tables \ref{tab:r=4}, \ref{tab:r=5}, and \ref{tab:r=6} that $k_2\leq \ell_2$. For those $\ell_2-4\leq k_2\leq \ell_2$, we rewrite 
\begin{equation}\label{eqn:shiftn}
n=s+r(m-2)(k_1-k)+rkm+r(k_2-2k)=s+r(m-2)(k_1-k)+ rkP_m(2)+r(k_2-2k)P_m(1).
\end{equation}
Having chosen $m$ large enough so that $\ell_2\geq 14$, we see that for $k\leq 5$ we have 
\[
k_2-2k\geq \ell_2-4-2k\geq 0.
\]
We then choose $k:=\min(5,k_1)$. If $k=k_1$, then \eqref{eqn:shiftn} gives a representation of $n$ with $r-1+ k+(k_2-2k)=r-1+k_2-k\leq \ell$ variables. On the other hand, if $k=5$, then Lemma \ref{lem:conguniversal} may be employed to represent $r(m-2)(k_1-5)$ and we obtain a representation of $n$ in \eqref{eqn:shiftn} with $r-1+5+(k_2-5)\leq \ell$ variables. 

\begin{center}
\begin{table}[th]\caption{Bounds for $k_2\leq \ell_2$ in the $r=4$ case\label{tab:r=4}}
\begin{tabular}{|c|c|c|}
\hline
Interval of $n<4(m-2)$& $s\in S_4$ & $k_2$ bound\\
\hline
\hline
$0\leq n\leq m-4$& $\{0,1,2,3\}$ & $k_2\leq \frac{m-4}{4}$\\
\hline
$m-3\leq n\leq 2m-7$& $\{m+j:-3\leq j\leq 2\}$ & $k_2\leq \frac{2m-7-(m-1)}{4}=\frac{m-6}{4}$\\
\hline
\{2m-6,2m-5\}&\{2m-6,2m-5\} & $k_2=0$\\
\hline
 $2m-4$& $\{m+j:-1\leq j\leq 2\}$&$k_2\leq \frac{2m-4-(m-1)}{4}=\frac{m-3}{4}$\\
\hline
$2m-3\leq n\leq 3m-10$&$\{2m-5,2m-3,$ & $k_2\leq \frac{3m-10-(2m-5)}{4}=\frac{m-5}{4}$\\
& $2m-2,2m,2m+1\}$&\\
\hline
\hline
$3m-9\leq n\leq 3m-4$&$\{3m+j:-9\leq j\leq -6\}$ & $k_2\leq 1$\\
\hline
$3m-3\leq n\leq 4m-9$&$\{3m+j:-3\leq j\leq 0\}$&$k_2\leq \frac{4m-9-(3m-3)}{4}=\frac{m-6}{4}$ \\ 
\hline
\end{tabular}
\end{table}
\end{center}

\begin{center}
\begin{table}[th]\caption{Bounds for $k_2\leq \ell_2$ in the $r=5$ case\label{tab:r=5}}
\begin{tabular}{|c|c|c|}
\hline
Interval of $n<5(m-2)$& $s\in S_5$ & $k_2$ bound\\
\hline
\hline
$0\leq n\leq m-4$& $\{0,1,2,3,4\}$ & $k_2\leq \frac{m-4}{5}$\\
\hline
$m-3\leq n\leq 2m-7$& $\{m+j:-3\leq j\leq 3\}$ & $k_2\leq \frac{2m-7-(m-1)}{5}=\frac{m-6}{5}$\\
\hline
$2m-6\leq n\leq 3m-10$&$\{2m+j:-6\leq j\leq 2\}$&$k_2\leq \frac{3m-10-(2m-2)}{5}=\frac{m-8}{5}$\\
\hline
$3m-9\leq n\leq 4m-11$&$\{3m+j:-9\leq j\leq -6$ & $k_2\leq \frac{4m-11-(3m-3)}{5}=\frac{m-8}{5}$\\
&or $-3\leq j\leq 1$ & \\
\hline
$4m-12\leq n\leq 5m-11$&$\{4m+j:-12\leq j\leq 0\}$&$k_2\leq \frac{5m-11-(4m-4)}{5}=\frac{m-7}{5}$ \\ 
\hline
\end{tabular}
\end{table}
\end{center}

\begin{center}
\begin{table}[th]\caption{Bounds for $k_2\leq \ell_2$ in the $r=6$ case\label{tab:r=6}}
\begin{tabular}{|c|c|c|}
\hline
Interval of $n<6(m-2)$& $s\in S_6$ & $k_2$ bound\\
\hline
\hline
$0\leq n\leq m-4$& $\{0,1,2,3,4,5\}$ & $k_2\leq \frac{m-4}{6}$\\
\hline
$m-3\leq n\leq 2m-7$& $\{m+j:-3\leq j\leq 4\}$ & $k_2\leq \frac{2m-7-m}{6}=\frac{m-7}{6}$\\
\hline
$2m-6\leq n\leq 3m-10$&$\{2m+j:-6\leq j\leq 3\}$ & $k_2\leq \frac{3m-10-(2m-1)}{6}=\frac{m-9}{6}$\\
\hline
$3m-9\leq n\leq 4m-11$&$\{3m+j:-9\leq j\leq 2\}$ & $k_2\leq \frac{4m-11-(3m-2)}{6}=\frac{m-9}{6}$\\
\hline
$4m-12\leq n\leq 5m-11$&$\{4m+j:-12\leq j\leq 1\}$ & $k_2\leq \frac{5m-11-(4m-3)}{6}=\frac{m-8}{6}$\\
\hline
$5m-12\leq n\leq 6m-16$&$\{5m+j:-12\leq j\leq -11$ & $k_2\leq \frac{6m-16-(5m-9)}{6}=\frac{m-2}{6}$\\
&or $-9\leq j\leq -6$ &\\
&or $-3\leq j\leq 0\}$ & \\
\hline
$6m-15\leq n\leq 6m-13$ &$\{6m+j:-15\leq j\leq -13\}$& $k_2=0$\\
\hline
\end{tabular}
\end{table}
\end{center}

\end{proof}
We are now ready to conclude the corollary.
\begin{proof}[Proof of Corollary \ref{cor:gammabound}]
The proof of Theorem \ref{thm:lmr} (3) immediately implies Corollary \ref{cor:gammabound} because either $3m-12$ or $2m-9$ is not represented by $P_{m,\bm{a}_{3,2,\ell}}$ for $\ell<\ell_{m,3,2}$, but one can see that every smaller integer is represented by $P_{m,\bm{a}_{3,2,\ell_{m,3,2}-1}}$. 
\end{proof}
\end{section}


\bigskip
\begin{thebibliography}{99}
\bibitem{BKSSV}M. Batavia, M. Kyranbay, S. Saha, H. C. So, and P. Varyani, \begin{it}Fermat's polygonal number theorem with repeats\end{it}, in preparation.
\bibitem{Bhargava}M. Bhargava, \begin{it}On the Conway--Schneeberger Fifteen Theorem\end{it}, Contemp. Math. \textbf{272} (2000), 27--38.
\bibitem{BhargavaHanke} M. Bhargava and J. Hanke, \begin{it}Universal quadratic forms and the $290$-theorem\end{it}, Invent. Math., to appear. 
\bibitem{BosmaKane}W. Bosma and B. Kane, \begin{it}The triangular theorem of eight and representations by quadratic polynomials\end{it}, Proc. Amer. Math. Soc. \textbf{141} (2013), 1473--1486.
\bibitem{Cauchy} A.-L. Cauchy, \begin{it}D\'emonstration du th\'eor\`em g\'en\'eral de Fermat sur les nombres polygones\end{it}, M\'em. Sci. Math. Phys. Inst. France \textbf{14} (1813--1815), 177--220; Oeuvres compl\`etes \textbf{VI} (1905), 320--353.
\bibitem{Conway}J. H. Conway, \begin{it}Universal quadratic forms and the fifteen theorem\end{it}, Contemp. Math. \textbf{272} (2000), 23--26.
\bibitem{Guy} R. K. Guy, \begin{it}Every number is expressible as the sum of how many polygonal numbers?\end{it}, Amer. Math. Monthly \textbf{101} (1994), 169--172.
\bibitem{Ju} J. Ju, \begin{it}Universal sums of generalized pentagonal numbers\end{it}, Ramanujan J. \textbf{51} (2020), 479--494.
\bibitem{JuOh} J. Ju and B.-K. Oh, \begin{it}Universal sums of generalized octagonal numbers\end{it}, J. Number Theory \textbf{190} (2018), 292--302.
\bibitem{KaneLiu}B. Kane and J. Liu, \begin{it}Universal sums of $m$-gonal numbers\end{it}, Int. Math. Res. Not., to appear, DOI: \url{https://doi.org/10.1093/imrn/rnz003}.
\bibitem{KimPark}B. M. Kim and D. Park, \begin{it}A finiteness theorem for universal $m$-gonal forms\end{it}, in preparation. 
\bibitem{Siegel1}C. Siegel, \begin{it}Indefinite quadratische Formen und Funktionentheorie, I\end{it}, Math. Ann. \textbf{124} (1951), 17--54.
\bibitem{Siegel2}C. Siegel, \begin{it}Indefinite quadratische Formen und Funktionentheorie, II\end{it}, Math. Ann. \textbf{124} (1951), 364--387.
\bibitem{Sun} Z.-W. Sun, \begin{it}A result similar to Lagrange's Theorem\end{it}, J. Number Theory \textbf{162} (2016), 190--211.
\bibitem{Weil}A. Weil, \begin{it}Sur la formule de Siegel dans la th\'eorie des groupes classiques\end{it}, Acta. Math. \textbf{113} (1965), 1--87.
\bigskip

\end{thebibliography}
\end{document}